\documentclass[UTF8,a4paper,11pt,reqno]{amsart}
\usepackage[left=2.50cm, right=2.50cm, top=2.50cm, bottom=2.50cm]{geometry} %页边距

\usepackage[pagewise]{lineno}%\linenumbers
\usepackage{latexsym}
\usepackage{amssymb}
\usepackage{newlfont}
\usepackage{hyperref}
\usepackage{mathrsfs}
\usepackage{amsmath}
\usepackage{amsthm}
\usepackage{booktabs}
\usepackage{amsfonts}
\usepackage{mathrsfs}

\newtheorem{thm}{Theorem}[section]
\newtheorem{cor}[thm]{Corollary}

\newtheorem{Def}[thm]{Definition}
\newtheorem{lemma}[thm]{Lemma}
\newtheorem{remark}[thm]{Remark}
\numberwithin{equation}{section}
\newcommand{\n}{\Vert}
\newcommand{\e}{\mathcal{E}}

\begin{document}

\title{\textbf{Long-time dynamics of the wave equation with nonlocal weak damping and sup-cubic nonlinearity in 3-D domains}}

\author[Senlin Yan]{Senlin Yan}

\address{Department of Mathematics, Nanjing University,
	Nanjing, Jiangsu, 210093, P.R. China} \email{dg20210019@smail.nju.edu.cn}

\author[Chengkui Zhong]{Chengkui Zhong$^*$}

\address{Department of Mathematics, Nanjing University,
	Nanjing, Jiangsu, 210093, P.R. China} \email{ckzhong@nju.edu.cn}

\author[Zhijun Tang]{Zhijun Tang}

\address{Department of Mathematics, Nanjing University,
	Nanjing, Jiangsu, 210093, P.R. China} \email{tzj960629@163.com}


\thanks{$^*$ Corresponding author.}

\begin{abstract}
	In this paper, we study the long-time dynamics for the wave equation with nonlocal weak damping and sup-cubic nonlinearity in a bounded smooth domain of $\mathbb{R}^3.$ Based on the Strichartz estimates for the case of bounded domains, we first prove the global well-posedness of the Shatah-Struwe solutions. Then we establish the existence of the global attractor for the Shatah-Struwe solution semigroup by the method of contractive function. Finally, we verify the existence of a polynomial attractor for this semigroup.
\end{abstract}

\maketitle

\section{Introduction}
	In this paper, we consider the following damped wave equation:
	\begin{equation}\label{eq1.1}
		\begin{cases}
			u_{tt}-\Delta u+k\n u_t\n^pu_t+f(u)=g,\quad (x,t)\in\Omega\times\mathbb{R}^+,\\
			u|_{\partial\Omega}=0,\quad t\in\mathbb{R}^+,\\
			u|_{t=0}=u_0,\ u_t|_{t=0}=u_1,\quad x\in\Omega.
		\end{cases}
	\end{equation}
	Here, $\Omega\subset\mathbb{R}^3$ is a bounded domain with smooth boundary, $\Vert\cdot\Vert$ stands for the usual norm of $L^2(\Omega)$, $k,p>0$ are fixed constants, $g\in L^2(\Omega)$ is independent of time, and the nonlinearity $f\in C^1(\mathbb{R})$ satisfies the following growth and dissipation conditions:
	\begin{align}
		\label{1.2}&|f'(s)|\leq C(1+|u|^{4-\kappa}),\quad s\in\mathbb{R},\\
		\label{1.3}&\liminf_{|s|\rightarrow\infty}f'(s)\equiv\mu>-\lambda_1,
	\end{align}
	where $C$ is a positive constant , $0<\kappa\leq4$ and $\lambda_1>0$ is the first eigenvalue of $-\Delta$ on $L^2(\Omega)$ with Dirichlet boundary conditions.
    
    The long-time dynamics of wave equations with nonlocal damping terms has been investigated extensively by many authors (see e.g. \cite{Bala,Cava,Silva1,Silva2,Silva3}). Recently, in \cite{Zhao3}, the authors studied the long-time behavior of a beam model with the nonlocal weak damping $\n u_t\n^pu_t$, and in \cite{Zhao1,Zhao2}, they also considered 
    the following wave equation with nonlocal weak damping and anti-damping in $\mathbb{R}^n(n\geq3)$
    \begin{equation}\label{zhao}
    	u_{tt}-\Delta u+k\n u_t\n^pu_t+f(u)=\int_{\Omega}K(x,y)u_t(y)dy+h(x),
    \end{equation}
    where $\int_{\Omega}K(x,y)u_t(y)dy$ is the anti-damping term and the growth exponent $q$ of the nonlinearity $f$ is up to the range: $0<q\leq\frac{n}{n-2}$ (in the case $n=3$, $0<q\leq3$). They obtained the well-posedness of problem (\ref{zhao}) in the natural energy space $H_0^1(\Omega)\times L^2(\Omega)$ and proved the existence of global attractors. It is worth mentioning that the nonlocal weak damping $\n u_t\n^pu_t$ was first proposed.
    
     We are motivated by \cite{Zhao1} to study problem (\ref{eq1.1}) where the growth exponent $q$ of the nonlinearity $f$ is allowed to be sup-cubic: $3<q<5$ (see condition (\ref{1.2})). Differently from Eq. (\ref{zhao}), since we are more interested in the nonlinearity with higher growth exponent, the anti-damping term is not considered in (\ref{eq1.1}).
    
    Similarly to the weakly damped wave equation, the long-time dynamics of problem (\ref{eq1.1}) depends strongly on the growth rate of nonlinearity $f$. One can verify the global well-posedness of weak energy solutions (solutions in the space $H_0^1(\Omega)\times L^2(\Omega)$) to problem (\ref{eq1.1}) only for the cases when the growth rate of $f$ is cubic or sub-cubic. Fortunately, with the help of suitable versions of Strichartz estimates and Morawetz-Pohozhaev identity in the case of bounded domains, one can obtain the global well-posedness of so-called Shatah-Struwe solutions (see Definition \ref{SS}) with sup-cubic nonlinearity $f$. Recently, Kalantarov et al. \cite{Zelik1} proved the existence and regularities of the compact global attractor of the Shatah-Struwe solution semigroup arising from weakly damped wave equation:
    \[u_{tt}+\gamma u_t-\Delta u+f(u)=g,\]
	where the nonlinearity $f(u)$ is up to quintic growth. After that, Liu et al. \cite{Liu1} considered the weakly damped wave equation with the source term $g\in H^{-1}$ and they obtained the existence of a global attractor as well as exponential attractor when the growth exponent of the nonlinearity satisfies $1<q<\frac{25}{7}$. Furthermore, in \cite{Savostianov,Savostianov2,Savostianov3}, the authors analyzed a more general situation, with fractional damping term $(-\Delta)^\theta u_t,\ \theta\in[0,1]$ in place of $u_t$. In this paper, we are devoted to study the long-time dynamics of Shatah-Struwe solutions to problem (\ref{eq1.1}).
	
	We want to point out here that the nonlocal weak damping and the sup-cubic nonlinearity together bring us difficulties to prove the well-posedness.
	 Since the function $\n u_t\n$ is not weakly continuous in $L^2$, it is difficult to use Galerkin's method directly to obtain the existence of solutions. Noticing that $\n u_t\n^pu_t$ is a monotone operator in $L^2(\Omega)$, the authors in \cite{Zhao1} used Theorem 1.5 in \cite{Chueshov2008} to obtain the well-posedness. However, in the sup-cubic case, $f(u):H^1_0\rightarrow L^2$ is no longer locally Lipschitz, thus Theorem 1.5 in \cite{Chueshov2008} is invalid. Inspired by \cite{Yang}, we combine the Galerkin's method and the monotonicity of $\n u_t\n^pu_t$ and finally obtain the well-posedness of problem (\ref{eq1.1}).
	
	The polynomial attractor is a concept to quantitatively describe
	the attractive velocity of attractors. A polynomical attractor is a positively invariant compact set which attracts any bounded set at a polynomial rate (see Definition \ref{poly}). In \cite{Zhao2,Zhao5}, the authors introduced the concept of $\varphi$-attractor for the first time and established some criteria for the existence of a polynomial attractor. Moreover, they applyed these abstract theorems to Eq. (\ref{zhao}) and obtained the existence of polynomial attractors with sub-critical ($q<\frac{n}{n-2}$) or critical ($q=\frac{n}{n-2}$) nonlinearity $f$. In this paper, we will use the abstract theorem established in \cite{Zhao5} to prove the existence of a polynomial attractor of the Shatah-Struwe solution semigroup generated by problem (\ref{eq1.1}).
	
	The rest of this paper is organized as follows. In Section 2, we give some preliminary things including Strichartz estimates and some abstract results. In Section 3, we discuss the global well-posedness of Shatah-Struwe solutions. In Section 4, we prove the existence of the global attractor for the corresponding dynamical system. Finally, the existence of a polynomial attractor is established in Section 6.

\section{Preliminaries}
	We first introduce the following abbreviations:
	\[L^p=L^p(\Omega),\quad H^1_0=H^1_0(\Omega),\quad \n\cdot\n=\n\cdot\n_{L^2},\]
	with $p\geq1.$ The notation $(\cdot,\cdot)$ stands for the usual inner product in $L^2.$ $X\hookrightarrow Y$ denotes that the space $X$ continuously embeds into $Y$ and $X\hookrightarrow\hookrightarrow Y$ denotes that $X$ compactly embeds into $Y.$ Let
	\[\e=H_0^1\times L^2,\]
	equipped with the usual graph norm:
	\[\n\xi_u\n_\e^2=\n\nabla u\n^2+\n u_t\n^2,\quad \xi_u:=(u,u_t)\in\e.\]

	We now recall the Strichartz estimates for wave equations in bounded domains. (see \cite{Blair,Burq1,Burq2} for details).
	\begin{equation}\label{eq2.1}
		\begin{cases}
			u_{tt}-\Delta u=G(t),\quad (x,t)\in\Omega\times\mathbb{R}^+,\\
			u|_{\partial\Omega}=0,\quad t\in\mathbb{R}^+,\\
			\xi_u(0)=\xi_0,\quad x\in\Omega.
		\end{cases}
	\end{equation}
	\begin{lemma}\cite{Blair}\label{lemma2.1}
		Let $\xi_0\in\e,\ G\in L^1(0,T;L^2)$ and $u(t)$ be a solution of equation (\ref{eq2.1}) such that $\xi_u\in C(0,T;\e)$. Then $u\in L^4(0,T;L^{12})$ and the following estimate holds:
		\begin{equation}
		\n u\n_{L^4(0,T;L^{12})}\leq C_T(\n\xi_0\n_\e+\n G\n_{L^1(0,T;L^{2})}),
		\end{equation}
		where the constant $C_T$ may depend on $T$, but is independent of $\xi_0$ and $G$.
	\end{lemma}
	We also state some lemmas, which are indispensable for our proof.
	\begin{lemma}\cite[Lemma 4.2]{Zhao5}\label{Zhao5}
		Let $(H,\langle\cdot,\cdot\rangle)$ be a an inner product space and $\n\cdot\n_H$ be the norm induced by $\langle\cdot,\cdot\rangle$. For any constant $p>1$, there exists $c_p>0$, such that $\forall x,y\in H$, $\n x\n_H\cdot\n y\n_H\neq0$, 
		\[\langle\n x\n_H^{p-2}x-\n y\n_H^{p-2}y,x-y\rangle\geq
		\begin{cases}
			C_p\n x-y\n_H^p,\quad if\ p\geq2;\\
			C_p\frac{\n x-y\n_H^2}{(\n x\n_H+\n y\n_H)^{2-p}},\quad if\ 1<p<2.
		\end{cases}\]
	\end{lemma}
	\begin{lemma}\cite{Savostianov}\label{Savostianov}
		Suppose that a function $y(s)\in C([a,b))$ satisfies $y(a)=0,\ y(s)\geq0$ and 
		\[y(s)\leq C_0y(s)^\sigma+\varepsilon\]
		for some $\sigma>1,\ 0<C_0<\infty$ and $0<\varepsilon<\frac{1}{2}(\frac{1}{2C_0})^\frac{1}{\sigma-1}$. Then
		\[y(s)\leq2\varepsilon,\quad \forall s\in[a,b).\]
	\end{lemma}
	\begin{lemma}(Aubin-Dubinskii-Lions Lemma,\cite{Lions})\label{Aubin}
		Assume that $X\subset Y\subset Z$ is a triple of Banach spaces such that $X\hookrightarrow\hookrightarrow Y\hookrightarrow Z$.\\
		(i) Let $W_1=\{u\in L^p(a,b;X)|u_t\in L^q(a,b;Z)\}$ for some $1\leq p\leq\infty$ and $q\geq1$. Here $u_t$ denotes the derivative of $u$ in the distributional sense. Then, $W_1\hookrightarrow\hookrightarrow L^p(a,b;Y).$ If $q>1$, then $W_1\hookrightarrow\hookrightarrow C(a,b;Z).$\\
		(ii) Let $W_2=\{u\in L^\infty(a,b;X)|u_t\in L^r(a,b;Z)\}$ for some $r>1$, then $W_2\hookrightarrow\hookrightarrow C(a,b;Y)$.
	\end{lemma}

\section{Well-posedness}
    The aim of this section is to study the well-posedness of the Shatah-Struwe solutions of problem (\ref{eq1.1}). We start with the definition of weak solutions.
	\begin{Def}
		For any $T>0$, a function $u(t),\ t\in[0,T]$ is said to be a weak solution of problem (\ref{eq1.1}) if $\xi_u\in L^\infty(0,T;\e)$ and Eq. (\ref{eq1.1}) is satisfied in the sense of distribution, i.e.
		\[-\int_{0}^{T}(u_t,\phi_t)dt+k\int_{0}^{T}\n u_t\n^p(u_t,\phi)dt+\int_{0}^{T}(\nabla u,\nabla\phi)dt+\int_{0}^{T}(f(u),\phi)dt=\int_{0}^{T}(g,\phi)dt,\]
		for any $\phi\in C_c^\infty((0,T)\times\Omega)$.
	\end{Def}
	To the best of our knowledge, the uniqueness of the weak solution to (\ref{eq1.1}) is still unknown when the nonlinearity satisfies (\ref{1.2}). Thus, following \cite{Zelik1}, we introduce the so-called Shatah-Struwe solutions.
	\begin{Def}\label{SS}
		A weak solution $u(t),\ t\in[0,T]$ is a  Shatah-Struwe solution of problem (\ref{eq1.1}) if the following additional regularity holds:
		\[u\in L^4(0,T;L^{12}).\]
	\end{Def}
	We first prove the local existence of Shatah-Struwe solutions.
	\begin{thm}\label{thm3.3}
		Let the nonlinearity $f$ satisfy assumptions (\ref{1.2}) and (\ref{1.3}). Then for any initial data $\xi_0\in\e$, there exists $T=T(\n\xi_0\n_\e)>0$ such that problem (\ref{eq1.1}) admits a Shatah-Struwe solution $u(t)$ on the interval $t\in[0,T]$ and the following estimate holds:
		\begin{equation}\label{3-3-0}
		      \n\xi_u(t)\n_\e+\n u\n_{L^4(0,T;L^{12})}\leq Q(\n\xi_0\n_\e),\quad t\in[0,T]
		\end{equation}
	    for some monotone function $Q$ which is independent of $u$.
	\end{thm}
	\begin{proof}
		Let $F(s)=\int_{0}^{s}f(\tau)d\tau$. Choose $\mu_0\in\mathbb{R}^+\cap(-\mu,\lambda_1)$. By (\ref{1.3}), there exists $M>0$ such that
		\[\begin{cases}
			F(s)\geq -\frac{\lambda_1+\mu_0}{4}s^2-C,\ &|s|\geq M;\\
			|F(s)|\leq C,\ &|s|\leq M.
		\end{cases}\]
	Then
		\begin{equation}\label{3-3-1}
			\int_{\Omega}F(u)dx\geq-\frac{\lambda_1+\mu_0}{4}\int_{\Omega}|u|^2dx-C\geq-\frac{\lambda_1+\mu_0}{4\lambda_1}\n \nabla u\n^2-C.
		\end{equation}
	    \indent We first assume that $\n\xi_0\n_\e\leq R$ and obtain some a priori estimates to the solutions of problem (\ref{eq1.1}). Multiplying (\ref{eq1.1}) by $u_t$ in $L^2$ yields
		\begin{equation}\label{3-3-2}
			\frac{d}{dt}E(u,u_t)+k\n u_t\n^{p+2}=0,
		\end{equation}
		where 
		\[E(u,u_t)=\frac{1}{2}\n \xi_u\n_\e^2+\int_{\Omega}F(u)dx-(g,u).\]
		By (\ref{1.2}), integrating (\ref{3-3-2}) over $(0,t)$ gives
		\begin{equation}\label{3-3-3}
			E(u(t),u_t(t))\leq E(u_0,u_1)\leq \frac{1}{2}\n\xi_0\n_\e^2+C(1+\n u_0\n_{L^6}^6)+\frac{1}{2}(\n u_0\n^2+\n g\n^2)\leq C_R,\quad t\geq0.
		\end{equation}
		Using Poincar\'{e}'s inequality and Young inequality, we have
		\begin{equation}\label{3-3-4}
			|(g,u)|\leq \frac{\lambda_1-\mu_0}{8\lambda_1}\n\nabla u\n^2+C(\n g\n)
		\end{equation}
		Thus, we can deduce from (\ref{3-3-1}) and (\ref{3-3-4}) that
		\begin{equation}\label{3-3-5}
			E(u,u_t)\geq \frac{\lambda_1-\mu_0}{8\lambda_1}\n\xi_u\n_\e^2-C.
		\end{equation}
		Combining (\ref{3-3-3}) and (\ref{3-3-5}), we have
		\begin{equation}\label{3-3-6}
			\n\xi_u(t)\n_\e\leq C_R,\quad t\geq0.
		\end{equation}
		\indent We will use the so-called Galerkin method. Let $\lambda_1\leq \lambda_2\leq\cdots$ be the eigenvalues of the operator $-\Delta$ with homogeneous Dirichlet boundary conditions on $\Omega$ and $e_1,e_2,\cdots$ be the corresponding eigenfunctions. Let $P_N:L^2\rightarrow L^2$ be the orthoprojector to the linear subspace $span\{e_1,e_2,\cdots,e_N\}$. Then, the Galerkin approximations to problem (\ref{eq1.1}) are definded as follows:
		\begin{equation}\label{eq3.3}
			\begin{cases}
				u^N_{tt}-\Delta u^N+k\n u^N_t\n^pu^N_t+P_Nf(u^N)=P_Ng,\quad u^N\in P_N L^2,\\
				\xi_{u^N}(0)=\xi_0^N:=P_N\xi_0.
			\end{cases}
		\end{equation}
		Obviously, Eq. (\ref{eq3.3}) is locally uniquely solvable and the solutions $u^N(t,x)$ are $C^\infty$-smooth in $x$. Note that for all $N$, $\n\xi_0^N\n_\e\leq\n\xi_0\n_\e\leq R$. Thus estimate (\ref{3-3-6}) holds for $u^N$, that is
		\begin{equation}\label{3-3-8}
			\n \xi_{u^N}(t)\n_\e\leq C_R,\quad t\geq0.
		\end{equation}
		\indent In order to estimate the $L^4(L^{12})-$norm of $u^N$, we split the solution $u^N=v^N+w^N$ where $v^N$ solves
		\begin{equation}\label{eqv}
			v^N_{tt}-\Delta v^N+k\n u^N_t\n^pv^N_t=P_Ng,\quad \xi_{v^N}(0)=\xi_0^N,
		\end{equation}
		and $w^N$ solves
		\begin{equation}\label{eqw}
			w^N_{tt}-\Delta w^N+k\n u^N_t\n^pw^N_t=-P_Nf(v^N+w^N),\quad \xi_{w^N}(0)=0.
		\end{equation}
		 Multiplying (\ref{eqv}) by $v^N_t$ in $L^2$ yields
		\[\frac{1}{2}\left(\n\xi_{v^N}\n_\e^2+2(g,v^N_t)\right)+k\n u^N_t\n^p\cdot\n v^N_t\n^2=0,\]
		Integrating this equation over $(0,t)$ and using Young inequality, we have
		\begin{equation}\label{3-3-11}
			\n\xi_{v^N}(t)\n_\e\leq \sqrt{3}\n\xi_{v^N}(0)\n_\e+C(\n g\n)\leq C_R,\quad t\geq0.
		\end{equation}
		Due to Lemma \ref{lemma2.1} and estimates (\ref{3-3-8}) and (\ref{3-3-11}), we get for any $t\leq 1,$
		\begin{equation}\label{3-3-12}
			\begin{split}
				\n v^N\n_{L^4(0,t;L^{12})}&\leq C\left(\n\xi_0^N\n_\e+\int_{0}^t(\n P_Ng\n+k\n u^N_t\n^p\cdot\n v^N_t\n) d\tau\right)\\
				&\leq C(\n\xi_0\n_\e+t\cdot\n g\n+t\cdot kC_R)\\
				&\leq C_R
			\end{split}
		\end{equation}
		Then, using the interpolation inequality, we have
		\begin{equation}\label{3-3-13}
			\n v^N\n_{L^5(0,t;L^{10})}\leq C\n v^N\n^{\frac{4}{5}}_{L^4(0,t;L^{12})}\cdot\n v^N\n^\frac{1}{5}_{L^\infty(0,t;H^1_0)}\leq C_R,\quad 0\leq t\leq1.
		\end{equation}
		For the variable $w^N$, estimates (\ref{3-3-8}) and (\ref{3-3-11}) imply
		\begin{equation}\label{3-3-14}
			\n\xi_{w^N}(t)\n_{\e}\leq \n\xi_{u^N}(t)\n_{\e}+\n\xi_{v^N}(t)\n_{\e}\leq C_R,\quad t\geq0.
		\end{equation}
		Due to the growth condition (\ref{1.2}) and estimate (\ref{3-3-13}), we have for $0\leq t\leq1,$
		\begin{align*}
			\n P_Nf(v^N+w^N)\n_{L^1(0,t;L^{2})}&\leq\int_{0}^{t}\n f(v^N+w^N)\n d\tau\\
			&\leq C\int_{0}^{t}\left(1+ \left\n |v^N|^{5-\kappa}+|w^N|^{5-\kappa}\right\n\right) d\tau\\
			&\leq C\int_{0}^{t}(1+ \n v^N\n_{L^{10}}^{5-\kappa}+ \n w^N\n_{L^{10}}^{5-\kappa})d\tau\\
			&\leq C(t+t^\frac{\kappa}{5}\n v^N\n_{L^5(0,t;L^{10})})+C\n w^N\n^5_{L^5(0,t;L^{10})}\\
			&\leq C(t+t^\frac{\kappa}{5}C_R)+C\n w^N\n^5_{L^5(0,t;L^{10})}.
		\end{align*}
		Applying now Lemma \ref{lemma2.1} to Eq. (\ref{eqw}) and using (\ref{3-3-14}) and the interpolation inequality, we end up with
		\begin{equation}
			\begin{split}
				\n w^N\n_{L^4(0,t;L^{12})}&\leq C\left(\n P_Nf(v^N+w^N)\n_{L^1(0,t;L^{2})}+k\int_{0}^{t}\n u^N_t\n^p\cdot\n w^N_t\n) d\tau\right)\\
				&\leq C(t+t^\frac{\kappa}{5}C_R+t\cdot kC_R)+C\n w^N\n^5_{L^5(0,t;L^{10})}\\
				&\leq C(t+t^\frac{\kappa}{5}C_R+tC_R)+C_R\n w^N\n^4_{L^4(0,t;L^{12})}.
			\end{split}
		\end{equation}
		For any $\varepsilon>0$, take $T=T(\varepsilon,R)\leq 1$, such that
		\[C(t+t^\frac{\kappa}{5}C_R+tC_R)\leq\varepsilon,\quad \forall t\leq T.\]
		Let $Y_N(t)=\n w^N\n_{L^4(0,t;L^{12})}$, then $Y_N(t)$ is continuous on $t$ and 
		\[Y_N(t)\leq \varepsilon+C_R(Y_N(t))^4,\quad \forall t\in[0,T],\quad Y_N(0)=0.\]
		By Lemma \ref{Savostianov}, if we choose $\varepsilon<\frac{1}{2}(\frac{1}{2C_R})^\frac{1}{3}$, then
		\begin{equation}\label{3-3-16}
			Y_N(t)\leq 2\varepsilon,\quad \forall t\in[0,T].
		\end{equation}
		Combining estimates (\ref{3-3-8}), (\ref{3-3-12}) and (\ref{3-3-16}), we get for some $T=T(R)$,
		\begin{equation}\label{3-3-17}
			\n\xi_{u^N}\n_{L^\infty(0,T;\e)}+\n u^N\n_{L^4(0,T;L^{12})}\leq C_R.
		\end{equation}
		Denote $\Omega_T=\Omega\times[0,T]$. From growth condition (\ref{1.2}), we also have
		\begin{align*}
			&\n f(u^N)\n_{L^\frac{6}{5}(\Omega_T)}\leq C\left(\int_{0}^{T}(1+\n u^N\n^6_{L^6})d\tau\right)^\frac{5}{6}\leq C(T(1+C_R))^\frac{5}{6},\\
			&\left\n k\n u^N_t\n^pu_t^N\right\n_{L^2(\Omega_T)}\leq k\left(\int_{0}^{T}\n u^N_t\n^{2p+2}d\tau\right)^\frac{1}{2}\leq k(TC_R)^\frac{1}{2}.
		\end{align*}
		\indent Now, by Alaoglu's theorem and Lemma \ref{Aubin}, we can extract a subsequence, still denoted by $\{u^N\}$, such that
		\begin{align*}
			&\xi_{u^N}\rightarrow\xi_u\ weakly^*\ in\ L^\infty(0,T;\e)\ and\ strongly\ in\ C(0,T;L^2\times H^{-1});\\
			&u^N\rightarrow u\ weakly\ in\ L^4(0,T;L^{12});\\
			&f(u^N)\rightarrow f(u) \ weakly\ in\ L^\frac{6}{5}(\Omega_T)\ and \ P_Nf(u^N)\rightarrow f(u)\ weakly\ in\ L^\frac{6}{5}(\Omega_T);\\
			&k\n u^N_t\n^pu_t^N\rightarrow \eta\ weakly\ in\ L^2(\Omega_T).
		\end{align*}

        From Eq. (\ref{eq3.3}), we also have $u^N_{tt}\rightarrow u_{tt}$ $weakly^*$ in $L^\infty(0,T;H^{-1})$. Passing the limit $N\rightarrow\infty$ in Eq. (\ref{eq3.3}), we get
        \begin{equation}\label{3-3-18}
        	u_{tt}-\Delta u+\eta+f(u)=g,\quad in\ H^{-1}\ for\ a.e.\ t\in[0,T].
        \end{equation}
		We claim that $\eta=k\n u_t\n^pu_t$ in $L^2(\Omega_T)$, and hence $u(t),t\in[0,T]$ is a Shatah-Struwe solution of problem (\ref{eq1.1}) and estimate (\ref{3-3-0}) follows from (\ref{3-3-17}) and the lower semi-continuity of the norm of weakly and weakly* limit.\\
		\indent Indeed, for any $\varphi\in L^2(\Omega_T)$, by Lemma \ref{Zhao5},
		\begin{equation}\label{3-3-19}
			\int_{0}^{T}(k\n u^N_t\n^pu_t^N-k\n\varphi\n^p\varphi,u^N_t-\varphi)d\tau\geq\int_{0}^{T}kC_p\n u^N_t-\varphi\n^{p+2}d\tau\geq0.
		\end{equation}
		Obviously, (\ref{3-3-2}) holds for $u^N$, i.e.
		\begin{equation}\label{energyN}
			\frac{d}{dt}E(u^N,u^N_t)+k\n u^N_t\n^{p+2}=0,\ a.e.\ t\in[0,T],
		\end{equation} 
	    and we need to verify that the solution $u$ also satisfies a similar equality, that is
		\begin{equation}\label{energy}
			\frac{d}{dt}E(u,u_t)+(\eta,u_t)=0, \ a.e.\ t\in[0,T].
		\end{equation}
		 Since $u\in L^\infty(0,T;H_0^1)\cap L^4(0,T;L^{12})$, by growth condition (\ref{1.2}) and the interpolation inequality, we have
		\[\n f(u)\n_{L^1(0,T;L^2)}\leq C(T+\n u\n^5_{L^5(0,T;L^{10})})\leq C(T+\n u\n^4_{L^4(0,T;L^{12})}\n u\n_{L^\infty(0,T;H^1_0)})\leq C_R.\]
		Then, $u_t\in L^\infty(0,T;L^2)$ implies that $f(u)u_t\in L^1([0,T]\times\Omega)$, and we can approximate $u$ by smooth function and argue in a standard way to deduce that for almost all $t\in[0,T],$
		\begin{equation}\label{3-3-22}
			\frac{d}{dt}\int_{\Omega}F(u)dx=(f(u),u_t).
		\end{equation}
		Let $u^n(t)=P_nu(t)$ where $P_n$ is the othoprojector on the the subspace  $span\{e_1,e_2,\cdots,e_n\}$. Then, $u^n$ solves
		\[u^n_{tt}-\Delta u^n+P_n\eta+P_nf(u)=P_ng.\]
		For every $0\leq \tau\leq t\leq T$, multiplying this equation by $u^n_t$ and integrating on $\Omega\times[\tau,t]$, we get
		\begin{equation}\label{3-3-23}
				\frac{1}{2}\n\xi_{u^n}(t)\n_\e^2-\frac{1}{2}\n\xi_{u^n}(\tau)\n_\e^2+\int_{\tau}^{t}(\eta,u^n_t)ds
				+\int_{\tau}^{t}(f(u),u^n_t)ds=(g,u^n(t))-(g,u^n(\tau)).
		\end{equation}
	    Note that $\xi_{u^n}(t)\rightarrow\xi_{u}(t)$, $\xi_{u^n}(\tau)\rightarrow\xi_{u}(\tau)$ in $\e$, $u^n_t\rightarrow u_t$ in $L^2$ strongly for $a.e.\ t\in[0,T]$ and $f(u)\in L^1(0,T;L^2)$, we may pass to the limit $n\rightarrow\infty$ in (\ref{3-3-23}) and use Lebesgue dominated convergence theorem to obtain that
	    \[\frac{1}{2}\n\xi_{u}(t)\n_\e^2-\frac{1}{2}\n\xi_{u}(\tau)\n_\e^2
	    +\int_{\tau}^{t}(f(u),u_t)ds-(g,u(t))+(g,u(\tau))+\int_{\tau}^{t}(\eta,u_t)ds=0.\]
	    With the help of (\ref{3-3-22}), we have
	    \[E(u(t),u_t(t))-E(u(\tau),u_t(\tau))+\int_{\tau}^{t}(\eta,u_t)ds=0\]
	    holds for every $0\leq\tau\leq t\leq T.$ Thus, (\ref{energy}) is proved.
		On the other hand, by (\ref{3-3-1}), the Fatou lemma and the lower semi-continuity of the norm of weakly* limit,
		\begin{align*}
			\int_{\Omega}\left(F(u)+\frac{\lambda_1+\mu_0}{4}|u|^2\right)dx&\leq \liminf_{N\rightarrow\infty}\int_{\Omega}\left(F(u^N)+\frac{\lambda_1+\mu_0}{4}| u^N|^2\right)dx\\
			&\leq \liminf_{N\rightarrow\infty}\int_{\Omega}F(u^N)dx+\frac{\lambda_1+\mu_0}{4}\n u\n^2,
		\end{align*}
		\begin{align*}
			\n\xi_{u^N}\n_\e^2\leq \liminf_{N\rightarrow\infty}\n\xi_{u}\n_\e^2.
		\end{align*}
		Now, it follows from (\ref{energyN}), (\ref{energy}) and above estimates,
		\begin{align*}
			\limsup_{N\rightarrow\infty}\int_{0}^{T}(k\n u^N_t\n^pu_t^N,u^N_t)d\tau&=\limsup_{N\rightarrow\infty}\left(E(u_0^N,u_1^N)-E(u^N(T),u_t^N(T))\right)\\
			&=E(u_0,u_1)-\liminf_{N\rightarrow\infty}E(u^N(T),u_t^N(T))\\
			&\leq E(u_0,u_1)-E(u(T),u_t(T))\\
			&=\int_{0}^{T}(\eta,u_t)d\tau,
		\end{align*}
		\begin{align*}
			\lim_{N\rightarrow\infty}\int_{0}^{T}(k\n u^N_t\n^pu_t^N,\varphi)d\tau=\int_{0}^{T}(\eta,\varphi)d\tau,
		\end{align*}
		\begin{align*}
			\lim_{N\rightarrow\infty}\int_{0}^{T}(k\n \varphi\n^p\varphi,u_t^N)d\tau=\int_{0}^{T}(k\n \varphi\n^p\varphi,u_t)d\tau.
		\end{align*}
		Then, by (\ref{3-3-19}),
		\[0\leq \limsup_{N\rightarrow\infty}\int_{0}^{T}(k\n u^N_t\n^pu_t^N-k\n\varphi\n^p\varphi,u^N_t-\varphi)d\tau\leq \int_{0}^{T}(\eta-k\n\varphi\n^p\varphi,u_t-\varphi)d\tau.\]
		Taking $\varphi=u_t-\lambda\psi$ with $\lambda>0$ and $\psi\in L^2(\Omega_T)$, we get
		\[\int_{0}^{T}(\eta-k\n u_t-\lambda\psi\n^p(u_t-\lambda\psi),\psi)d\tau\geq 0.\]
		Leting $\lambda\rightarrow0^+$, we end up with
		\[\int_{0}^{T}(\eta-k\n u_t\n^pu_t,\psi)d\tau\geq 0,\quad \forall\psi\in L^2(\Omega_T).\]
		This means that $\eta=k\n u_t\n^pu_t$ in $L^2(\Omega_T)$ and the proof is complete.		
	\end{proof}
	
	\begin{remark}\label{remark3.4}
		Since (\ref{energy}) holds and $\eta=k\n u_t\n^pu_t$, we know that a Shatah-Struwe solution $u(t),t\in[0,T]$ satisfies the energy equality:
		\begin{equation}\label{3-4}
			\frac{d}{dt}E(u,u_t)+k\n u_t\n^{p+2}=0,\quad a.e.\ t\in[0,T],
		\end{equation}
		where
		\[E(u,u_t)=\frac{1}{2}\n \xi_u\n^2+\int_{\Omega}F(u)dx-(g,u).\]
		In particular, $\xi_u\in C(0,T;\e)$. Indeed, from (\ref{3-4}), we know that $t\rightarrow\n\xi_u(t)\n_\e$ is continuous. On the other hand, since $\xi_u\in L^\infty(0,T;\e)\cap C(0,T;L^2\times H^{-1})$ and $\e$ is a reflexive space, we have $\xi_u\in C_w(0,T;\e)$. Then, the uniform convexity of $\e$ gives that $\xi_u\in C(0,T;\e).$
	\end{remark}

	\begin{cor}\label{cor}
		Let the assumptions of Theorem \ref{thm3.3} hold. Then, for any initial data $\xi_0\in\e$, there exists a unique global Shatah-Struwe solution $u(t)$ of problem (\ref{eq1.1}) and we have the following estimate
		\begin{equation}\label{3-5-0}
			\n\xi_u(t)\n_\e+\n u\n_{L^4(t,t+1;L^{12})}\leq Q(\n\xi_0\n_\e),\quad t\geq0,
		\end{equation}
	    where the monotone function $Q$ is independent of $u$ and $t$.
	\end{cor}
	\begin{proof}
		Suppose that $u(t), t\in[0,T]$ is a Shatah-Struwe solution with initial data $\xi_u(0)=\xi_0$. By Remark \ref{remark3.4}, $u$ satisfies the energy equality (\ref{3-4}). Integrating (\ref{3-4}) over $(0,t)$ yields
		\[E(u(t),u_t(t))\leq E(u_0,u_1)\leq Q_1(\n\xi_0\n_\e),\quad t\in[0,T]\]
		for some monotone function $Q_1$. Then, from (\ref{3-3-5}), we have
		\begin{equation}\label{3.26}
			\n\xi_{u}(t)\n_\e\leq Q_2(\n\xi_0\n_\e),\quad t\in[0,T]
		\end{equation}
	    for some monotone function $Q_2$ which is independent of $u$ and $t$. 
		According to this estimate, the energy norm $\n\xi_u(t)\n_\e$ cannot blow up in a finite time. Therefore, the local Shatah-Struwe solution $u(t)$ can be extended globally in time and the estimate (\ref{3-5-0}) follows from (\ref{3-3-0}) and (\ref{3.26}). We remain to show the uniqueness of Shatah-Struwe solution.\\
		\indent Let $u(t),v(t),t\in[0,T]$ be two Shatah-Struwe solutions of problem (\ref{eq1.1}) with initial data $\n\xi_u(0)\n_\e$, $\n\xi_v(0)\n_\e\leq R$. Then, 
		\[\n u\n_{L^4(0,T;L^{12})}+\n v\n_{L^4(0,T;L^{12})}\leq C(R,T).\]
		Let $w(t)=u(t)-v(t)$, then $w$ solves
		\[w_{tt}-\Delta w+k\n u_t\n^pu_t-\n v_t\n^pv_t+f(u)-f(v)=0.\]
		Multiplying this equation by $w_t$ and integrating over $x\in\Omega$(which can be justified according to Remark \ref{remark3.4}), we get
		\begin{equation}\label{3-5-1}
			\frac{1}{2}\frac{d}{dt}\n\xi_w(t)\n_\e^2+(k\n u_t\n^pu_t-\n v_t\n^pv_t,w_t)+(f(u)-f(v),w_t)=0.
		\end{equation}
		By Lemma \ref{Zhao5} and condition (\ref{1.2}),
		\begin{equation}\label{3-5-2}
			(k\n u_t\n^pu_t-\n v_t\n^pv_t,w_t)\geq kC_p\n u_t-v_t\n^{p+2}\geq0,
		\end{equation}
		\begin{equation}\label{3-5-3}
			\begin{split}
				|(f(u)-f(v),w_t)|&\leq C((1+|u|^4+|v|^4)|w|,|w_t|)\\
				&\leq C(1+\n u\n^4_{L^{12}}+\n v\n^4_{L^{12}})\n w\n_{L^6}\n w_t\n\\
				&\leq C(1+\n u\n^4_{L^{12}}+\n v\n^4_{L^{12}})\n \xi_w\n_\e^2.
			\end{split}
		\end{equation}
		Inserting (\ref{3-5-2}) and (\ref{3-5-3}) into (\ref{3-5-1}), we have
		\[\frac{1}{2}\frac{d}{dt}\n\xi_w\n_\e^2\leq C(1+\n u\n^4_{L^{12}}+\n v\n^4_{L^{12}})\n \xi_w\n_\e^2.\]
		Applying the Gronwall inequality, we get
		\[\n\xi_w(t)\n_\e^2\leq \n\xi_w(0)\n_\e^2\exp\left\{C\int_{0}^{t}(1+\n u\n^4_{L^{12}}+\n v\n^4_{L^{12}})d\tau\right\}\leq C(R,T)\n\xi_u(0)-\xi_v(0)\n_\e^2,\quad t\in[0,T].\]
		Therefore, $\xi_w(t)\equiv0$ on $[0,T]$ if the initial data $\xi_u(0)=\xi_v(0)$ and the proof is complete.
	\end{proof}
	
	We now define the solution semigroup $S(t):\e\rightarrow\e$ associated with problem (\ref{eq1.1}):
	\[S(t)\xi_0:=\xi_u(t),\]
	where $u(t)$ is the unique Shatah-Struwe solution $u(t)$ of problem (\ref{eq1.1}) corresponding to initial data $\xi_0$. Moreover, according to the proof of Corollary \ref{cor}, $S(t):\e\rightarrow\e$ is locally Lipschitz continuous.

\section{Global attractor}
    
    \begin{thm}\label{thm4.1}
    	Under the conditions (\ref{1.2}) and (\ref{1.3}), the dynamical system $(\e,S(t))$ is dissipative.
    \end{thm}
    \begin{proof}
    	Following exactly from the proof process of Theorem 3.1 in \cite{Zhao1}, where we take $n=3$, $K(x,y)\equiv0$ and replace the growth condition (1.4) in \cite{Zhao1} with condition (\ref{1.2}) in this paper, we can obtain the dissipativity of $(\e,S(t))$.
    \end{proof}
    
     Now, we are ready to verify the asymptotical compactness of $(\e,S(t))$ by applying the contractive function method.  
    \begin{Def}\cite{Chueshov2015,Chueshov2008}
    	Let $X$ be a Banach space, $B\subset X$ be a bounded subset. A function $\Psi:X\times X\rightarrow\mathbb{R}$ is said to be a contractive function on $B\times B$ if for any sequence $\{x_n\}_{n=1}^\infty\subset B$, there is a subsequence $\{x_{n_k}\}_{n=1}^\infty\subset \{x_n\}_{n=1}^\infty$ such that
    	\[\liminf_{k\rightarrow\infty}\liminf_{l\rightarrow\infty}\Psi(x_{n_k},x_{n_l})=0.\]
    	Denote the set of all contractive functions on $B\times B$ by $\mathcal{C}(B)$.
    \end{Def}
    
    \begin{lemma}\cite{Chueshov2015,Chueshov2008}\label{lem4.2}
    	Let $\{S(t)\}_{t\geq0}$ be a continuous semigroup on a Banach space $(X,\n\cdot\n)$. If for any positively invariant bounded set $B\subset X$ and for any $\varepsilon>0$, there exist $T=T(\varepsilon,B)$ and a contractive function $\Psi_{\varepsilon,T,B}\in\mathcal{C}(B)$ such that
    	\[\n S(T)y_1-S(T)y_2\n \leq \varepsilon+\Psi_{\varepsilon,T,B}(y_1,y_2),\quad \forall y_1,y_2\in B.\]
    	Then, $\{S(t)\}_{t\geq0}$ is asymptotically compact.
    \end{lemma}

    \indent To establish the main result in this section, we need the following result.
    \begin{lemma}\label{lem4.4}
    	Under the condition (\ref{1.2}), the operator
    	\begin{align*}
    		T_f:\mathcal{M}&\longrightarrow L^1(0,T;L^2)\\
    		&u\longmapsto f(u)
    	\end{align*}
    	is compact, where $\mathcal{M}=\{u\in L^\infty(0,T;H^1_0)\cap L^4(0,T;L^{12})|u_t\in L^\infty(0,T;L^2)\}$.
    \end{lemma}
    \begin{proof}
    	For any $u,v\in \mathcal{M}$, by the growth condition (\ref{1.2}) we have
%    	\begin{equation*}
 %   		\begin{split}
 %   			\n f(u)\n_{L^1(0,T;L^{2})}&\leq C\int_{0}^{T}\left(\int_{\Omega}(1+|u|^5)^2dx\right)^\frac{1}{2}dt\\
%    			&\leq C\int_{0}^{T}(1+\n u\n_{L^{10}}^5)dt\\
%    			&\leq CT+C\n u\n^5_{L^5(0,T;L^{10})}\\
 %   			&\leq CT+C\n u\n^4_{L^4(0,T;L^{12})}\n u\n_{L^\infty(0,T;H^1_0)},
 %   		\end{split}
%    	\end{equation*}
    	\begin{align*}
    		\n f(u)-f(v)\n_{L^1(0,T;L^{2})}&=\int_{0}^{T}\n (f(u)-f(v))\n dt\\
    		&\leq C\int_{0}^{T}\left\n (1+|u|^{4-\kappa}+|v|^{4-\kappa})|u-v|\right\n dt\\
    		&\leq C\int_{0}^{T}\left\n 1+|u|^{4-\kappa}+|v|^{4-\kappa}\right\n_{L^\frac{12}{4-\kappa}}\n u-v\n_{L^\frac{12}{\kappa+2}} dt\\
    		&\leq C\int_{0}^{T}\left(1+\n u\n^{4-\kappa}_{L^{12}}+\n v\n^{4-\kappa}_{L^{12}}\right)\n u-v\n_{L^\frac{12}{\kappa+2}}dt\\
    		&\leq C(T+\n u\n_{L^4(0,T;L^{12})}+\n v\n_{L^4(0,T;L^{12})})\cdot\sup_{t\in[0,T]}\n u-v\n_{L^\frac{12}{\kappa+2}}.
    	\end{align*}
    	Since $2\leq\frac{12}{\kappa+2}<6$, we have $H_0^1\hookrightarrow\hookrightarrow L^\frac{12}{\kappa+2}\hookrightarrow L^2$. Using Lemma \ref{Aubin}, we know that $\mathcal{M}\hookrightarrow\{u\in L^\infty(0,T;H^1_0)|u_t\in L^\infty(0,T;L^2)\}\hookrightarrow\hookrightarrow C(0,T;L^\frac{12}{\kappa+2})$. Therefore, $T_f$ is a compact operator.
    \end{proof}
 
    \begin{thm}\label{thm4.5}
    	Under the conditions (\ref{1.2}) and (\ref{1.3}), the dynamical system $(\e,S(t))$ is asymptotically compact.
    \end{thm}
    \begin{proof}
    	Let $B\subset\e$ be a positively invariant bounded set.\\
    	\indent\textbf{A priori estimates.} For any $\xi,\xi'\in B$, let $\xi_u(t)=S(t)\xi,\ \xi_v(t)=S(t)\xi',$ and
    	 $z(t)=u(t)-v(t)$. Then $z(t)$ satisfies 
    	\begin{equation}\label{4-5-1}
    		z_{tt}-\Delta z+k\n u_t\n^pu_t-k\n v_t\n^pv_t+f(u)-f(v)=0,\ \xi_z(0)=\xi-\xi'.
    	\end{equation}
    	Denote 
    	\[E_z(t)=\frac{1}{2}\n\xi_z(t)\n_\e^2.\]
    	Multiplying (\ref{4-5-1}) by $z_t$ and integrating over $x\in\Omega$,
    	\begin{equation}\label{4-5-2}
    		\frac{d}{dt}E_z(t)+(k\n u_t\n^pu_t-k\n v_t\n^pv_t,z_t)+(f(u)-f(v),z_t)=0
    	\end{equation}
    	Integrating (\ref{4-5-2}) on $[t,T]$,
    	\begin{equation}\label{4-5-3}
    		E_z(T)+k\int_{t}^{T}(\n u_t\n^pu_t-\n v_t\n^pv_t,z_t)d\tau=E_z(t)-\int_{t}^{T}(f(u)-f(v),z_t)d\tau.
    	\end{equation}
    	Integrating (\ref{4-5-3}) on $[0,T]$ with respect to $t$,
    	\begin{equation}\label{4-5-4}
    		TE_z(T)+k\int_{0}^{T}\int_{t}^{T}(\n u_t\n^pu_t-\n v_t\n^pv_t,z_t)d\tau dt=\int_{0}^{T}E_z(t)dt-\int_{0}^{T}\int_{t}^{T}(f(u)-f(v),z_t)d\tau dt.
    	\end{equation}
    	Multiplying (\ref{4-5-1}) by $z$ and integrating over $x\in\Omega,t\in[0,T]$,
    	\begin{equation}\label{4-5-5}
    		\begin{split}
    			\int_{0}^{T}E_z(t)dt&=-\frac{1}{2}(z_t,z)|_0^T-\frac{k}{2}\int_{0}^{T}(\n u_t\n^pu_t-\n v_t\n^pv_t,z)dt\\
    			&\quad+\int_{0}^{T}\n z_t\n^2dt-\frac{1}{2}\int_{0}^{T}(f(u)-f(v),z)dt\\
    			&\leq C_B+TC_B\sup_{t\in[0,T]}\n z(t)\n+\int_{0}^{T}\n z_t\n^2dt+\frac{1}{2}\left|\int_{0}^{T}(f(u)-f(v),z)dt\right|.
    		\end{split}
    	\end{equation}
    	Integrating (\ref{4-5-2}) on $[0,T]$,
    	\begin{equation}\label{4-5-6}
    		\begin{split}
    			\int_{0}^{T}(k\n u_t\n^pu_t-k\n v_t\n^pv_t,u_t-v_t)dt&=E_z(0)-E_z(T)-\int_{0}^{T}(f(u)-f(v),z_t)dt\\
    			&\leq C_B+\left|\int_{0}^{T}(f(u)-f(v),z_t)dt\right|.
    		\end{split}
    	\end{equation}
    	Using Young's inequality with $\varepsilon$ and Lemma \ref{Zhao5},
    	\begin{equation}\label{4-5-7}
    		\begin{split}
    			\n z_t\n^2&\leq \varepsilon+C_\varepsilon\n u_t-v_t\n^{p+2}\\
    			&\leq \varepsilon+C_\varepsilon(k\n u_t\n^pu_t-k\n v_t\n^pv_t,u_t-v_t).
    		\end{split}
    	\end{equation}
    	Combining (\ref{4-5-6}) and (\ref{4-5-7}), we have
    	\begin{equation}\label{4-5-8}
    		\int_{0}^{T}\n z_t\n^2dt\leq \varepsilon T+C_{\varepsilon,B}+C_\varepsilon\left|\int_{0}^{T}(f(u)-f(v),z_t)dt\right|.
    	\end{equation}
    	Inserting (\ref{4-5-5}) and (\ref{4-5-8}) into (\ref{4-5-4}) and using Lemma \ref{Zhao5}, 
    	\begin{align*}
    		TE_z(T)&\leq \varepsilon T+C_{\varepsilon,B}+TC_B\sup_{t\in[0,T]}\n z(t)\n+\frac{1}{2}\left|\int_{0}^{T}(f(u)-f(v),z)dt\right|\\
    		&+C_\varepsilon\left|\int_{0}^{T}(f(u)-f(v),z_t)dt\right|+\left|\int_{0}^{T}\int_{t}^{T}(f(u)-f(v),z_t)d\tau dt\right|.
    	\end{align*}
         Set
        \begin{equation}\label{4-5-9}
        	\begin{split}
        		\Psi_T(\xi,\xi')&=T\sup_{t\in[0,T]}\n u(t)-v(t)\n+\left|\int_{0}^{T}(f(u)-f(v),u-v)dt\right|\\
        		&+\left|\int_{0}^{T}(f(u)-f(v),u_t-v_t)dt\right|+\left|\int_{0}^{T}\int_{t}^{T}(f(u)-f(v),u_t-v_t)d\tau dt\right|,
        	\end{split}
        \end{equation}
    	then we obtain that 
    	\begin{equation}
    		E_z(T)\leq \varepsilon+\frac{C_{\varepsilon,B}}{T}+\frac{C_{\varepsilon,B}}{T}\Psi_{T}(\xi,\xi').
    	\end{equation}
    	\indent\textbf{Asymptotical compactness.} Due to Lemma \ref{lem4.2}, it suffices to prove that the function $\Psi_{T}$ defined in (\ref{4-5-9}) belongs to $\mathcal{C}(B)$ for any fixed $T>0$.\\
    	\indent For any sequence $\{\xi_n\}_{n=1}^\infty\subset B$, let $\xi_{u^n}(t)=S(t)\xi_n$. Since $B$ is positively invariant and $\{u^n(t)\}_{n=1}^\infty,t\in[0,T]$ are Shatah-Struwe solutions, we have
    	\begin{equation}\label{4-11}
    		\n\xi_{u^n}\n_{L^\infty(0,T;\e)}\leq C_B,\ \n u^n\n_{L^4(0,T;L^{12})}\leq C_{B,T},\ \forall n\in\mathbb{N},
    	\end{equation} 
    	For the sequence $\{\xi_{u^n}\}_{n=1}^\infty$, from (\ref{4-11}), Lemma \ref{Aubin} and Lemma \ref{lem4.4}, we deduce that there exists a subsequence, still denoted by $\{\xi_{u^n}\}$, such that
    	\begin{equation}\label{4-12}
    		\n f(u^n)\n_{L^1(0,T;L^{2})}\leq CT+C\n u^n\n^4_{L^4(0,T;L^{12})}\n u^n\n_{L^\infty(0,T;H^1_0)}\leq C_{B,T},
    	\end{equation}
        and
    	\begin{equation}\label{4-13}
    		\begin{cases}
    			\xi_{u^n}\rightarrow \xi_u\ weakly^*\ in\ L^\infty(0,T;\e),\\
    			u^n\rightarrow u\ weakly\ in\ L^4(0,T;L^{12}),\\
    			u^n\rightarrow u\ strongly\ in\ C(0,T;L^2)\cap C(0,T;L^\frac{12}{\kappa+2}),\\
    			f(u^n)\rightarrow f(u)\ strongly\ in\ L^1(0,T;L^2).
    		\end{cases}
    	\end{equation}
    	Now, we are ready to verify that $\Psi_{T}
    	\in \mathcal{C}(B)$. Due to (\ref{4-11}), (\ref{4-12}) and (\ref{4-13}), we get
    	\begin{equation}\label{4-14}
    		\lim_{n\rightarrow\infty}\lim_{m\rightarrow\infty}\sup_{t\in[0,T]}\n u^n(t)-u^m(t)\n=0,
    	\end{equation}
    	\begin{equation}\label{4-15}
    		\begin{split}
    			&\quad\lim_{n\rightarrow\infty}\lim_{m\rightarrow\infty}\int_{0}^{T}(f(u^n)-f(u^m),u^n-u^m)dt\\
    			&\leq \lim_{n\rightarrow\infty}\lim_{m\rightarrow\infty}\n f(u^n)-f(u^m)\n_{L^1(0,T;L^{2})}\sup_{t\in[0,T]}\n u^n(t)-u^m(t)\n\\
    			&\leq C_{B,T}\lim_{n\rightarrow\infty}\lim_{m\rightarrow\infty}\sup_{t\in[0,T]}\n u^n(t)-u^m(t)\n=0,
    		\end{split}
    	\end{equation}
        \begin{equation}\label{4-16}
        	\begin{split}
        		&\quad\lim_{n\rightarrow\infty}\lim_{m\rightarrow\infty}\int_{t}^{T}(f(u^n)-f(u^m),u^n_t-u^m_t)d\tau\\
        		&\leq\lim_{n\rightarrow\infty}\lim_{m\rightarrow\infty}\n f(u^n)-f(u^m)\n_{L^1(t,T;L^{2})}\n u_t^n-u_t^m\n_{L^\infty(t,T;L^2)}\\
        		&\leq C_B\lim_{n\rightarrow\infty}\lim_{m\rightarrow\infty}\n f(u^n)-f(u^m)\n_{L^1(0,T;L^{2})}=0,\quad\forall t\in[0,T].
        	\end{split}
        \end{equation}
    	Moreover, by (\ref{4-11}) and (\ref{4-12}),
    	\begin{equation}\label{4-17}
    		\left|\int_{t}^{T}(f(u^n)-f(u^m),u^n_t-u^m_t)d\tau\right|\leq C_{B,T},\ \forall t\in[0,T].
    	\end{equation}
    	According to Lebesgue dominated convergence theorem, we infer from (\ref{4-16}) and (\ref{4-17}) that 
    	\begin{equation}\label{4-18}
    		\lim_{n\rightarrow\infty}\lim_{m\rightarrow\infty}\int_{0}^{T}\int_{t}^{T}(f(u^n)-f(u^m),u^n_t-u^m_t)d\tau dt=0.
    	\end{equation}
    	Combining (\ref{4-5-9}), (\ref{4-14}), (\ref{4-15}), (\ref{4-16}) and (\ref{4-18}) that
    	\[\lim_{n\rightarrow\infty}\lim_{m\rightarrow\infty}\Psi_{T}(\xi_n,\xi_m)=0.\]
    	Hence, $\Psi_{T}\in\mathcal{C}(B)$. The proof is complete.
    \end{proof}

    \begin{thm}\label{thm4.6}
    	Under the conditions (\ref{1.2}) and (\ref{1.3}), the dynamical system $(\e,S(t))$ generated by the Shatah-Struwe solutions of problem (\ref{eq1.1}) possesses a global attractor. 
    \end{thm}
    \begin{proof}
    	From Corollary \ref{cor}, the Shatah-Struwe solution semigroup $S(t):\e\rightarrow\e$ is well-defined and continuous for any fixed $t\geq0.$ Then, by the standard abstract results about the existence of attractor (see e.g. \cite{Babin,Hale,Temam}), we can conclude that Theorem \ref{thm4.6} is an immediate consequence of Theorem \ref{thm4.1} and Theorem \ref{thm4.5}.
    \end{proof}

\section{Polynomical attractor}
    In this section, we will show that the dynamical system $(\e,S(t))$ possesses a polynomial attractor. We begin with the definition.
    \begin{Def}\cite[Definition 2.6]{Zhao5}\label{poly}
    	Let $\{S(t)\}_{t\geq0}$ be a continuous semigroup on a complete metric space $(X,d)$. Assume that $\varphi:\mathbb{R}^+\rightarrow\mathbb{R}^+$ satisfies $\varphi(t)\rightarrow0$ as $t\rightarrow\infty.$ A compact set $\mathcal{A}^*\subset X$ is said to be a generalized $\varphi$-attractor for the dynamical system $(X,S(t))$ if\\
    	(i) $\mathcal{A}^*$ is positively invariant, i.e. $S(t)\mathcal{A}^*\subset\mathcal{A}^*,\ \forall t\geq0$;\\
    	(ii) there exists $t_0\in\mathbb{R}$ such that for any bounded set $B\subset X$ there exists $t_B\geq0$ such that
    	\[\mathrm{dist}(S(t)B,\mathcal{A}^*)\leq \varphi(t+t_0-t_B),\quad \forall t\geq t_B,\]
    	where $\mathrm{dist}(A,B):=\sup_{x\in A}\inf_{y\in B}d(x,y)$ is the Hausdorff semidistance. In particular, if $\varphi(t)=Ct^{-\beta}$ for certain positive constants $C,\beta$, then $\mathcal{A}^*$ is called a generalized polynomial attractor; if $\varphi(t)=Ce^{-\beta t}$ for certain positive constants $C,\beta$, then $\mathcal{A}^*$ is called a generalized exponential attractor.
    \end{Def}

    \begin{lemma}\cite[Theorem 3.1]{Zhao5}\label{lem5.2}
    	Let $\{S(t)\}_{t\geq0}$ be a dissipative continuous semigroup on a complete metric space $(X,d)$ and $\mathcal{B}_0$ be a positively invariant bounded absorbing set. Assume that there exist positive constants $C,T,\delta_0,\beta\in(0,1)$, functions $g_l:(\mathbb{R}^+)^m\rightarrow\mathbb{R}^+(l=1,2)$ and pseudometrics $\varrho_T^i(i=1,2,\cdots,m)$ on $\mathcal{B}_0$ such that\\
    	(i) $g_l$ is non-decreasing with respect to each variable, $g_l(0,\cdots,0)=0$ and $g_l$ is continuous at $(0,\cdots,0)$;\\
    	(ii) $\varrho_T^i(i=1,2,\cdots,m)$ is precompact on $\mathcal{B}_0$, i.e. any sequence $\{x_n\}\subset\mathcal{B}_0$ has a subsequence $\{x_{n_k}\}$ which is Cauchy with respect to $\varrho_T^i$;\\
    	(iii) for any $y_1,y_2\in\mathcal{B}_0$ satisfying $\varrho_T^i(y_1,y_2)\leq\delta_0(i=1,2,\cdots,m)$, the following inequalities holds:
    	\begin{equation}
    		(d(S(T)y_1,S(T)y_2))^2\leq (d(y_1,y_2))^2+g_1(\varrho_T^1(y_1,y_2),\varrho_T^2(y_1,y_2),\cdots,\varrho_T^m(y_1,y_2)),
    	\end{equation}
        \begin{equation}
        	\begin{split}
        		(d(S(T)y_1,S(T)y_2))^2\leq&C[d(y_1,y_2)^2-(d(S(T)y_1,S(T)y_2))^2\\
        		&+g_1(\varrho_T^1(y_1,y_2),\varrho_T^2(y_1,y_2),\cdots,\varrho_T^m(y_1,y_2))]^\beta\\
        		&+g_2(\varrho_T^1(y_1,y_2),\varrho_T^2(y_1,y_2),\cdots,\varrho_T^m(y_1,y_2)).
        	\end{split}
        \end{equation}
        Then, the dynamical system $(X,S(t))$ possesses a generalized polynomial attractor $\mathcal{A}^*$ and there exists $t_0>0$ such that
        \begin{equation}
        	\mathrm{dist}(S(t)B,\mathcal{A}^*)\leq \left\{(\alpha(\mathcal{B}_0))^\frac{2(\beta-1)}{\beta}+\frac{1-\beta}{T\beta(1+2C)^\frac{1}{\beta}}(t-t_0-2T-t_*(B)-1)\right\}^\frac{\beta}{2(\beta-1)}
        \end{equation}
        holds for every bounded set $B\subset X$, where $t_*(B)$ satisfies $S(t)B\subset \mathcal{B}_0(\forall t\geq t_*(B))$ and $\alpha(\mathcal{B}_0)$ denotes the noncompactness measure of $\mathcal{B}_0$.
    
    \end{lemma}

    We are ready to state the main result in this section.
    \begin{thm}
    	Under the conditions (\ref{1.2}) and (\ref{1.3}), the dynamical system $(\e,S(t))$ generated by the Shatah-Struwe solutions of problem (\ref{eq1.1}) possesses a generalized polynomial attractor $\mathcal{A}^*$ and there exists $t_0>0$ such that
    	\[\mathrm{dist}(S(t)B,\mathcal{A}^*)\leq \left\{(\alpha(\mathcal{B}_0))^{-p}+\frac{pkC_p}{2^{p+2}}(t-t_0-t_*(B)-1)\right\}^{-\frac{1}{p}}.\]
    	holds for every bounded set $B\subset X$, where $\mathcal{B}_0\subset\e$ is a positively invariant bounded absorbing set of $(\e,S(t))$ and  $t_*(B)$ satisfies $S(t)B\subset \mathcal{B}_0(\forall t\geq t_*(B))$.  
    \end{thm}
    \begin{proof}
    	By Theorem \ref{thm4.1}, the dynamical system $(\e,S(t))$ is dissipative and let $\mathcal{B}_0\subset\e$ be a positively invariant bounded absorbing set of $(\e,S(t))$. 
    	
    	For any $y_1,y_2\in\mathcal{B}_0$, let $\xi_u(t)=S(t)y_1, \xi_v(t)=S(t)y_2$ and $z(t)=u(t)-v(t)$. Let $T$ be an arbitrary positive constant, then since $\mathcal{B}_0$ is positively invariant and $u(t),v(t)$ are Shatah-Struwe solutions, we have
    	\begin{equation}\label{5-4}
    		\begin{cases}
    			\n\xi_u\n_{L^\infty(\mathbb{R}^+;\e)}\leq C,\ \n\xi_v\n_{L^\infty(\mathbb{R}^+;\e)}\leq C;\\
    			\n u\n_{L^4(0,T;L^{12})}\leq C_{T},\ \n v\n_{L^4(0,T;L^{12})}\leq C_{T}.
    		\end{cases}
    	\end{equation}
        The function $z(t)$ satisfies
        \begin{equation}
        	z_{tt}-\Delta z+k\n u_t\n^pu_t-k\n v_t\n^pv_t+f(u)-f(v)=0,\ \xi_z(0)=y_1-y_2.
        \end{equation}
        Denote 
        \[E_z(t)=\frac{1}{2}\n\xi_z(t)\n_\e^2.\]
        Arguing as in the "\textbf{A priori estimates}" part of Theorem \ref{thm4.5}, we can infer from (\ref{4-5-3}), (\ref{4-5-4}) and (\ref{4-5-5}) that
        \begin{equation}\label{5-6}
        	E_z(T)+k\int_{0}^{T}(\n u_t\n^pu_t-\n v_t\n^pv_t,z_t)dt=E_z(0)-\int_{0}^{T}(f(u)-f(v),z_t)dt,
        \end{equation}
        \begin{equation}\label{5-7}
        	\begin{split}
        		TE_z(T)&+k\int_{0}^{T}\int_{t}^{T}(\n u_t\n^pu_t-\n v_t\n^pv_t,z_t)d\tau dt\\
        		&=-\frac{1}{2}(z_t,z)|_0^T-\frac{k}{2}\int_{0}^{T}(\n u_t\n^pu_t-\n v_t\n^pv_t,z)dt+\int_{0}^{T}\n z_t\n^2dt\\
        		&\quad-\frac{1}{2}\int_{0}^{T}(f(u)-f(v),z)dt-\int_{0}^{T}\int_{t}^{T}(f(u)-f(v),z_t)d\tau dt.
        	\end{split}
        \end{equation}
        By Lemma \ref{Zhao5}, 
        \[(\n u_t\n^pu_t-\n v_t\n^pv_t,u_t-v_t)\geq C_p\n u_t-v_t\n^{p+2},\]
        which, together with the concavity of the function $h(s)=s^\frac{2}{p+2}(s>0)$, yields
        \begin{equation}\label{5-8}
        	\begin{split}
        		\int_{0}^{T}\n z_t\n^2dt&\leq C_p^{-\frac{2}{p+2}}\int_{0}^{T}(\n u_t\n^pu_t-\n v_t\n^pv_t,u_t-v_t)^\frac{2}{p+2}dt\\
        		&\leq C_p^{-\frac{2}{p+2}}T^\frac{p}{p+2}\left(\int_{0}^{T}(\n u_t\n^pu_t-\n v_t\n^pv_t,z_t)dt\right)^\frac{2}{p+2}.
        	\end{split}
        \end{equation}
        Combining (\ref{5-6}) with (\ref{5-8}), we have
        \begin{equation}\label{5-9}
        	\begin{split}
        		\int_{0}^{T}\n z_t\n^2dt\leq (kC_p)^{-\frac{2}{p+2}}T^\frac{p}{p+2}\left(E_z(0)-E_z(T)-\int_{0}^{T}(f(u)-f(v),z_t)dt\right)^\frac{2}{p+2}.
        	\end{split}
        \end{equation}
        By the proof of Lemma \ref{lem4.4} and the estimate (\ref{5-4}), 
        \begin{align*}
        	\n f(u)-f(v)\n_{L^1(0,T;L^{2})}&\leq C(T+\n u\n_{L^4(0,T;L^{12})}+\n v\n_{L^4(0,T;L^{12})})\sup_{t\in[0,T]}\n u(t)-v(t)\n_{L^\frac{12}{\kappa+2}}\\
        	&\leq C_{T}\sup_{t\in[0,T]}\n z(t)\n_{L^\frac{12}{\kappa+2}}.
        \end{align*}
        Then, it is easy to obtain the following estimates:
        \[-\frac{1}{2}(z_t,z)|_0^T\leq C\sup_{t\in[0,T]}\n z(t)\n,\]
        \[-\frac{k}{2}\int_{0}^{T}(\n u_t\n^pu_t-\n v_t\n^pv_t,z)dt\leq C_T\sup_{t\in[0,T]}\n z(t)\n,\]
        \begin{align*}
        	-\frac{1}{2}\int_{0}^{T}(f(u)-f(v),z)dt&\leq \n f(u)-f(v)\n_{L^1(0,T;L^{2})}\cdot\n z\n_{L^\infty(0,T;L^2)}\\
        	&\leq C\n f(u)-f(v)\n_{L^1(0,T;L^{2})}\\
        	&\leq C_{T}\sup_{t\in[0,T]}\n z(t)\n_{L^\frac{12}{\kappa+2}},
        \end{align*}
        \begin{align*}
        	\int_{t}^{T}(f(u)-f(v),z_t)d\tau&\leq \n f(u)-f(v)\n_{L^1(t,T;L^{2})}\cdot\n z_t\n_{L^\infty(t,T;L^2)}\\
        	&\leq C\n f(u)-f(v)\n_{L^1(0,T;L^{2})}\\
        	&\leq C_{T}\sup_{t\in[0,T]}\n z(t)\n_{L^\frac{12}{\kappa+2}},\quad \forall t\in[0,T].
        \end{align*}
        Plugging these inequalities and (\ref{5-9}) into (\ref{5-7}) and using Lemma \ref{Zhao5}, we end up with
        \begin{equation}\label{5-10}
        	\begin{split}
        		E_z(T)&\leq C_T\left(\sup_{t\in[0,T]}\n z(t)\n+\sup_{t\in[0,T]}\n z(t)\n_{L^\frac{12}{\kappa+2}}\right)\\
        		&\quad+(kTC_p)^{-\frac{2}{p+2}}\left(E_z(0)-E_z(T)+C_T\sup_{t\in[0,T]}\n z(t)\n_{L^\frac{12}{\kappa+2}}\right)^\frac{2}{p+2}\\
        		&\leq C_T\sup_{t\in[0,T]}\n z(t)\n_{L^\frac{12}{\kappa+2}}+(kTC_p)^{-\frac{2}{p+2}}\left(E_z(0)-E_z(T)+C_T\sup_{t\in[0,T]}\n z(t)\n_{L^\frac{12}{\kappa+2}}\right)^\frac{2}{p+2},
        	\end{split}
        \end{equation}
        where the last less-than-equal comes from the embedding $L^\frac{12}{\kappa+2}\hookrightarrow L^2(\frac{12}{\kappa+2}\geq2)$. From (\ref{5-6}), we also have
        \begin{equation}\label{5-11}
        	E_z(T)\leq E_z(0)+\n f(u)-f(v)\n_{L^1(0,T;L^{2})}\cdot\n z_t\n_{L^\infty(0,T;L^2)}\leq E_z(0)+C_T\sup_{t\in[0,T]}\n z(t)\n_{L^\frac{12}{\kappa+2}}.
        \end{equation}
         Since $2\leq\frac{12}{\kappa+2}<6$, we have $H_0^1\hookrightarrow\hookrightarrow L^\frac{12}{\kappa+2}\hookrightarrow L^2$. Using Lemma \ref{Aubin}, we know that $\{u\in L^\infty(0,T;H^1_0)|u_t\in L^\infty(0,T;L^2)\}\hookrightarrow\hookrightarrow C(0,T;L^\frac{12}{\kappa+2})$. Thus, $\varrho_T(y_1,y_2)=\sup_{t\in[0,T]}\n z(t)\n_{L^\frac{12}{\kappa+2}}$ is precompact on $\mathcal{B}_0$. Then, by the definition of $E_z(t)$, we infer from (\ref{5-11}) and (\ref{5-10}), 
         \begin{equation}\label{5-12}
         	\n S(T)y_1-S(T)y_2\n_\e^2\leq \n y_1-y_2\n_\e^2+2C_T\varrho_T(y_1,y_2)
         \end{equation}
        and
        \begin{equation}\label{5-13}
        	\begin{split}
        		\n S(T)y_1-S(T)y_2\n_\e^2\leq&2(2kTC_p)^{-\frac{2}{p+2}}\left(\n y_1-y_2\n_\e^2-\n S(T)y_1-S(T)y_2\n_\e^2+2C_T\varrho_T(y_1,y_2)\right)^\frac{2}{p+2}\\
        		&+2C_T\varrho_T(y_1,y_2).
        	\end{split}
        \end{equation}
        By Lemma \ref{lem5.2}, we deduce from (\ref{5-12}) and (\ref{5-13}) that the dynamical system $(\e,S(t))$ possesses a generalized polynomial attractor $\mathcal{A}^*$ and there exists $t_0>0$ such that
        \begin{equation}\label{5-14}
        	\mathrm{dist}(S(t)B,\mathcal{A}^*)\leq \left\{(\alpha(\mathcal{B}_0))^{-p}+\frac{p}{2(T^\frac{2}{p+2}+2^\frac{2p+2}{p+2}(kC_p)^{-\frac{2}{p+2}})^\frac{p+2}{2}}(t-t_0-2T-t_*(B)-1)\right\}^{-\frac{1}{p}}
        \end{equation}
        holds for every bounded set $B\subset X$, where $t_*(B)$ satisfies $S(t)B\subset \mathcal{B}_0(\forall t\geq t_*(B))$. 
        Since $T>0$ is arbitrary, we can pass the limit $T\rightarrow0$ in (\ref{5-14}) and obtain that
        \[\mathrm{dist}(S(t)B,\mathcal{A}^*)\leq \left\{(\alpha(\mathcal{B}_0))^{-p}+\frac{pkC_p}{2^{p+2}}(t-t_0-t_*(B)-1)\right\}^{-\frac{1}{p}}.\]
        The proof is complete.
    \end{proof}

\section*{Acknowledgements}
We would like to express our sincere thanks to the everyone valuable comments and suggestions which play an important role of our original manuscript. This work was supported by the National Science Foundation of China Grant (11731005) .

\section*{Data Availability Statement}	
Data sharing is not applicable to this article as no new data were created or analyzed in this study.


\begin{thebibliography}{24}
	\bibliographystyle{model1-num-names}
	
	
	 \bibitem{Babin}
	A. Babin and M. Vishik, Attractors of Evolution Equations, North-Holland, Amsterdam, 1992.
	
	\bibitem{Bala}
	A.V. Balakrishnan and L.W. Taylor, Distributed parameter nonlinear damping models for flight structures, in: Proceedings Damping 89, Flight Dynamics Lab and Air Force Wright Aeronautical Labs, WPAFB, 1989.
	

	
	\bibitem{Blair}
    M. D. Blair, H. F. Smith and C. D. Sogge,
    Strichartz estimates for the wave equation on manifolds with boundary, \emph{Ann. Inst. H. Poincaré Anal. Non Linéaire}, \textbf{26} (2009), 1817-1829.

    \bibitem{Burq1}
    N. Burq, G. Lebeau and F. Planchon, Global existence for energy critical waves in 3-D domains, \emph{J. Amer. Math. Soc.}, \textbf{21} (2008), 831-845.

    \bibitem{Burq2}
    N. Burq and F. Planchon, Global existence for energy critical waves in 3-D domains: Neumann boundary conditions, \emph{Amer. J. Math.} \textbf{131} (2009),1715-1742.
	

	\bibitem{Cava}
	M. Cavalcanti, M.A.J. Silva and C. Webler, Exponential stability for the wave equation with degenerate nonlocal weak damping, \emph{Israel J. Math.} \textbf{219} (2017), 189-213.
	
	
	\bibitem{Chueshov2015}
	I. Chueshov, \emph{Dynamics of Quasi-stable Dissipative Systems}, Springer, Switzerland, 2015. 
	
	\bibitem{Chueshov2008}
	I. Chueshov and I. Lasiecka, \emph{Long-time Behavior of Second Order Evolution Equations with Nonlinear Damping}, Mem. Amer. Math. Soc., vol. 195, (Providence, RI: American Mathematical Society), 2008. 
	
	 \bibitem{Hale}
	J. Hale, Asymptotic Behavior of Dissipative Systems, \emph{American Mathematical Society}, Providence, RI, 1988.
	
	
	\bibitem{Zelik1}
	V. Kalantarov, A. Savostianov, S. Zelik, Attractors for damped quintic wave equations in bounded domains, \emph{Ann. Henri Poincaré}, \textbf{17} (2016),2555-2584.
	
	 \bibitem{Lions}
	J.L. Lions, \emph{Quelques méthodes de résolution des problèmes aux limites non linéaires}, Dunod, Paris, 1969.
	
	\bibitem{Liu1}
	C. Liu, F. Meng and C. Sun, Well-posedness and attractors for a super-cubic weakly damped wave equation with $H^{-1}$ source term, \emph{Journal of Differential Equations}, \textbf{263} (2017), 8718-8748.

	
    \bibitem{Savostianov}
    A. Savostianov, Strichartz estimates and smooth attractors for a sub-quintic wave equation with fractional damping in bounded domains, \emph{Adv. Differential Equations}, \textbf{20} (2015), 495-530.
	
	\bibitem{Savostianov2}
	A. Savostianov and S.V. Zelik, Smooth attractors for the quintic wave equations with fractional damping, Asymptot. Anal. \textbf{87} (2014), 191-221.
	
	\bibitem{Savostianov3}
	A. Savostianov and S.V. Zelik, Recent progress in attractors for quintic wave equations, Math. Bohem, \textbf{139} (2014), 657-665.

   \bibitem{Silva1}
   M.A.J. Silva and V. Narciso, Attractors and their properties for a class of nonlocal extensible beams, \emph{ Discrete Contin. Dyn. Syst.}, \textbf{35} (2015), 985-1008.
   
   \bibitem{Silva2}
   M.A.J. Silva and V. Narciso, Long-time dynamics for a class of extensible beams with nonlocal nonlinear damping, \emph{Evol. Equ. Control Theory}, \textbf{6} (2017), 437-470.
   
   \bibitem{Silva3}
   M.A.J. Silva, V. Narciso and A. Vicente, On a beam model related to flight structures with nonlocal energy damping, \emph{Discrete Contin. Dyn. Syst. Ser. B}, \textbf{24} (2019) 3281-3298.
   
   
   
   \bibitem{Temam}
   R. Temam, Infinite-Dimensional Dynamical Systems in Mechanic and Physics, 2nd edition. App. Math. Sci., vol. 68. Springer-Verlag, New York, 1997.
    
    \bibitem{Yang}
    Z. Yang, On an extensible beam equation with nonlinear damping and source terms, \emph{J. Differential Equations}, \textbf{254} (2013), 3903-3927.
    

    
    
    \bibitem{Zhao3}
    C.X. Zhao, C.Y. Zhao and C.K. Zhong. The global attractor for a class of extensible beams with nonlocal weak damping, \emph{Discrete Contin. Dyn. Syst. Ser. B}, \textbf{25} (2020), 935-955.
    
    
    \bibitem{Zhao1}
    C.Y. Zhao, C.X. Zhao and C.K. Zhong, Asymptotic behaviour of the wave equation with nonlocal weak damping and anti-damping, J. Math. Anal. Appl.,
    \textbf{490} (2020), 124186-124202. 
    
    \bibitem{Zhao2}
    C.Y. Zhao, C.K. Zhong and S.L. Yan, Existence of a generalized polynomial attractor for the wave equation with nonlocal weak damping, anti-damping and critical nonlinearity, \emph{Appl. Math. Lett.}, \textbf{128} (2022), 107791.
    
    
    
    \bibitem{Zhao5}
    C.Y. Zhao, C.K. Zhong and C.X. Zhao, Estimate of the attractive velocity of attractors for some dynamical systems (in Chinese), \emph{Scientia Sinica Mathematica}, \textbf{52} (2022), 1-20, doi: 10.1360/SCM-2021-0470.


   










\end{thebibliography}
\end{document}